\documentclass[11pt]{article}
\usepackage{amsmath, amssymb, amsfonts, amstext, amsthm, textcomp, enumerate}
\usepackage[mathscr]{euscript}
\usepackage{graphicx}
\usepackage{booktabs}
\usepackage{array}
\usepackage{lscape}
\usepackage{caption}
\usepackage{mathtools}
\newtheorem{thm}{Theorem}[section]
\newtheorem{lem}[thm]{Lemma}%[section]
\newtheorem{cor}[thm]{Corollary}%[section]
\newtheorem{defn}[thm]{Definition}%[section]
\newtheorem{rem}[thm]{Remark}%[section]
\newtheorem{ex}[thm]{Example}%[section]

\newtheorem{obs}[thm]{Observation}

\newcommand{\IM}{\mbox{$\mathbf{IM}$}}
\newcommand{\IDM}{\mbox{$\mathbf{IDM}$}}
\newcommand{\SIDM}{\mbox{$\mathbf{SIDM}$}}

\newcommand{\TP}{\mbox{$\mathbf{TP}$}}
\newcommand{\TN}{\mbox{$\mathbf{TN}$}}
\newcommand{\PSD}{\mbox{$\mathbf{PSD}$}}
\newcommand{\PD}{\mbox{$\mathbf{PD}$}}
\newcommand{\DN}{\mbox{$\mathbf{DN}$}}

\usepackage{color}

\newcommand{\mnr}{\mathbf{M}_n\,(\mathbb{R})}
\newcommand{\mnc}{\mathbf{M}_n\,(\mathbb{C})}

\newcommand{\diag}{{\rm diag}}

\begin{document}

\title{Powers, Roots, and Positivity}
\author{
    Charles R Johnson$^1$
    \and
    Samir Mondal$^2$
	\and
	Michael J. Tsatsomeros$^3$}
	    \footnotetext[1]{225 West Tazewells Way, Williamsburg, VA 23185,  (crjmatrix@gmail.com).}
		\footnotetext[2]{Department of Mathematics and Statistics,
University of Regina, 
3737 Wascana Parkway
Regina, SK S4S 0A2, 
Canada (isamirmondal@gmail.com).}
		\footnotetext[1]{Department of Mathematics and Statistics,
	                     Washington State University,
                     	Pullman, WA 99164 ({\tt tsat@wsu.edu}).
        
\date{}	}

\maketitle

\begin{abstract}

% There is a  remarkable variety of natural generalizations of the notion of scalar positivity to square matrices. Lists and tables indicate those that we consider here; they are practically exhaustive of that arise in practice and in theory. Positive scalars are closed under both powers and arbitrary roots, so that it is natural to ask which generalizations are also. Our purpose is 1) to survey all generalizations with respect to closure under powers and (fractional) roots, and 2) to indicate the role of pick functions in the analysis. Proofs and counterexamples are given, as well as natural related results, as well as tables summarizing the results. This should provide a helpful reference.

There is a remarkable variety of natural generalizations of the notion of scalar positivity to square matrices. The lists and tables indicate those that we consider here; they are practically exhaustive of those that arise in practice and theory. Since positive scalars are closed under taking arbitrary powers and roots, it is natural to ask which of these matrix generalizations enjoy analogous closure properties.

Our purpose here is to examine and advance these generalizations with respect to their closure under integer powers and fractional matrix roots. In addition to surveying known results, we establish for several important classes of matrices that positivity is preserved under fractional powers, showing that key structural features are retained under matrix roots. The role of Pick functions in the analysis is emphasized throughout. Proofs, counterexamples, and related results are presented, and summary tables are included to provide a concise and useful reference.
  
\end{abstract}


% 15A48: Positive matrices and their generalizations; cones of matrices
% 15A16: Matrix exponential and similar functions of matrices
% 15A18: Eigenvalues, singular values, and eigenvectors
% 15A23: Factorization of matrices
% 15A09: Matrix inversion, generalized inverses

%%%%%%%%%%%%%%%%%%%%%%%%%%%%%%%%%%%%%%%%%%%%%%%%%%%%%%%%%%%%%%%%%%%%%%%%%%%

\section{Introduction}

Matrix positivity is a fundamental concept in linear algebra with broad theoretical and practical importance. {It extends the notion of positivity from scalars to matrices. This extension has been approached in several ways, using conditions involving matrix entries, transformation properties, minors, quadratic forms, and combinations thereof.}
Key classes include positive definite matrices, positive stable matrices, matrices whose Hermitian part is positive definite, copositive matrices, totally positive (TP) matrices, \(P\)-matrices,  semipositive matrices, invertible \(M\)-matrices, inverse \(M\)-matrices, \(H\)-matrices, strictly diagonally dominant matrices, doubly nonnegative matrices, and related positivity classes.
 These classes arise in diverse fields such as data analysis, differential equations, mathematical programming, computational complexity, economic modeling, population biology, dynamical systems, and control theory. Positivity matrix classes exhibit strong structural properties that are often preserved under standard matrix operations. Among these operations, matrix powers and roots are particularly significant due to their roles in both theoretical investigations and practical applications, including in fields such as finance and healthcare \cite{HL}. When a matrix exhibits a specific structure, it becomes important to investigate the existence of a square root that maintains the same structure or demonstrates a related pattern. Understanding how these structures behave under matrix functions, especially integer and fractional powers, is essential for advancing theory and addressing challenges in applications like Markov modeling in which preserving the structural properties of the matrix, such as positivity, stochasticity, or spectral characteristics is crucial.

This line of inquiry fits naturally within the broader study of operations that preserve matrix structure. In particular, \cite{JohnsonSmith2006} studied closure under Schur complementation, including singular cases, for several important matrix classes, including symmetric, positive definite, Hessenberg, tridiagonal, and circulant matrices.

The study of matrix roots has attracted considerable attention, especially in preserving matrix positivity and related structures. Among various approaches, power series expansions and iterative schemes are the most prominent methods for computing matrix \( p \)th roots. Iterative methods such as Newton’s, Halley’s, and Schröder’s methods have been extensively analyzed for their convergence behavior as well as their ability to preserve matrix structures. It was shown in \cite{luguo2018} through power series analysis that for nonsingular \( M \)-matrices and \( H \)-matrices with positive diagonal entries, Newton’s and Halley’s methods consistently preserve these important matrix structures at every iteration and ensure smooth, monotone convergence to the principal \( p \)th root, which also remains within the same matrix class.
%\cite{luguo2018}. 
This framework was later extended to Schröder’s method in \cite{guolu}. Earlier, \cite{ianna} developed efficient and numerically stable iterative algorithms for computing matrix roots with a focus on numerical stability. The foundational theory of matrix functions, including structure-preserving properties, is comprehensively treated in \cite{hig}. Further study of \( p \)th roots of stochastic matrices using functional calculus was carried out in \cite{HL}, demonstrating conditions under which stochasticity and nonnegativity are preserved. Additionally, fractional powers of \( M \)-matrices and \( P \)-matrices were analyzed in \cite{ando,SST}, respectively, using analytic function theory, thereby contributing to the understanding of positivity preservation in this context.

% We study \textbf{fractional powers} and \textbf{roots} of matrices within key positivity classes. More precisely, let \( \mathcal{C} \) denote a matrix positivity class (such as the class of \( M \)-matrices, \( P \)-matrices, stochastic matrices, or totally positive matrices). Using techniques from Ando (1980) and Mondal, Sivakumar, and Tsatsomeros (2024), we consider matrix functions defined via \textbf{functional calculus}, particularly those induced by \textbf{Pick functions}, leveraging tools from \textbf{ analysis} and \textbf{matrix theory}.

% Our main problem can be stated as follows:

% \begin{center}
	%     \textbf{Given \( A \in \mathcal{C} \) and a Pick function \( f \),\\
		%     under what conditions does \( f[A] \in \mathcal{C} \) hold?}
	% \end{center}

% Since many functions of interest, including fractional powers \( f(z) = z^p \) with \( 0 < p < 1 \), are Pick functions on \( (0, \infty) \), this framework provides a natural setting to investigate \textbf{structure preservation under matrix fractional powers}, particularly \textbf{matrix roots}.

We study {fractional powers} and {integer powers} of matrices within key positivity classes. More precisely, let \( \mathcal{C} \) denote a matrix positivity class (such as the class of positive definite matrices, totally positive matrices, \( M \)-matrices, or \( P \)-matrices). {Here and throughout, by fractional matrix powers \(A^p\) we mean real powers with \(0<p<1\), where \(p\) is not necessarily rational; these powers are defined via the functional calculus described in Section~\ref{Prelims}.}

For fractional powers, we ask the following question:
\begin{center}
    \textbf{Given \(A\in\mathcal{C}\) and \(0<p<1\), is there an
    \(A^p\) in \(\mathcal{C}\)?}
\end{center}

For this purpose, we consider matrix functions $f[A]$ defined via {functional calculus}, focusing particularly on those induced by {Pick functions} (i.e., analytic functions 
that map the upper half-plane into itself), and leveraging tools from {analysis} and {matrix theory}. This leads naturally to the more general question:
% \begin{center}
% 	\textbf{Given \( A \in \mathcal{C} \) and a Pick function \( f \),\\
% 		under what conditions does \( f[A] \in \mathcal{C} \) hold?}
% \end{center}

\begin{center}
	\textbf{Given \( A \in \mathcal{C} \) and a Pick function \( f \),
	is \( f[A] \in \mathcal{C} \)?}
\end{center}

Since many functions of interest, including the fractional power function
$f(z)=z^p \; (0<p<1)$,
are Pick functions on (0$,\infty)$, this framework provides a natural setting for investigating structure preservation under fractional matrix powers and, in particular, matrix roots.

For completeness, we also investigate the behavior of positive integer powers:
\begin{center}
	\textbf{Given \( A \in \mathcal{C} \), is \( A^k \in \mathcal{C} \) for all \( k \in \mathbb{N} \)?}
\end{center}

Significant results have been established for Pick functions applied to invertible $M$-matrices, inverse $M$-matrices, and $P$-matrices, where all fractional powers remain within the same matrix class. For many other important positivity classes, however, the preservation of positivity under Pick functions, particularly under fractional powers, remains less well understood. In this paper, we investigate these questions for a broader collection of matrix classes, obtaining new preservation results, identifying limitations, and bringing together several related results from the literature within a unified framework.

The main contribution of this paper is the study of Pick functions that preserve positivity across a broad range of matrix classes, including:

\begin{enumerate}
	\item Positive stable matrices  
	\item Positive definite matrices  
	\item Matrices with Hermitian part positive definite  
	\item Copositive matrices  
	\item Semipositive matrices 
	\item Invertible real \( H \)-matrices and inverse \( H \)-matrices  
	\item Strictly generalized diagonally dominant matrices
	\item Doubly nonnegative matrices  
	
\end{enumerate}

% The main contribution of this paper is the study of Pick functions that preserve positivity across a broad range of matrix classes, including positive stable matrices, positive definite matrices, matrices with positive definite Hermitian part, copositive matrices, semipositive matrices, invertible real $H$-matrices and inverse $H$-matrices, strictly generalized diagonally dominant matrices, and doubly nonnegative matrices.

By employing Pick function theory, we demonstrate that principal fractional powers of matrices in each of these classes remain within the same class, thereby significantly advancing the theory of structure preservation under matrix functions.

The remainder of this paper is structured as follows: Section \ref{Prelims} contains notation and terminology and introduces the matrix classes under consideration. Some basic facts to be used are quoted with references. Also included in Section \ref{Prelims} are essential facts about Pick functions, matrix functions, the matrix exponential and logarithm, as well as roots of matrices. Section \ref{Powers} contains the main results on powers in matrix positivity classes. The main focus is on the preservation of positivity classes under fractional powers via Pick functions, followed by preservation under positive integer powers. In the concluding section, two summarizing tables are featured. The first table categorizes the matrix positivity classes according to preservation under fractional powers, integer powers, neither, or both. The second table summarizes the behavior of the matrix classes studied under Pick functions and fractional powers. In three instances, preservation by Pick functions remains open.

%%%%%%%%%%%%%%%%%%%%%%%%%%%%%%%%%%%%%%%%%%%%%%%%%%%%%%%%%%%%%%%%%%%%%%%

\section{Notation, Definitions, and Preliminaries}
\label{Prelims}

Throughout, our considerations are over the field of real numbers.
We denote by $\mathbf{M}_{m,n}(\mathbb{R})$ the space of all real $m \times n$ matrices,
and write $\mnr := \mathbf{M}_{n,n}(\mathbb{R})$.
Similar notation is used for complex matrices, e.g., $\mnc := \mathbf{M}_{n,n}(\mathbb{C})$.
We write $\sigma(A)$ and
$\rho(A):=\max\{|\lambda|:\lambda\in\sigma(A)\}$
for the spectrum and spectral radius of a matrix $A$, respectively, and
$\mathrm{GL}_n(\mathbb{R})$ for the group of all invertible matrices in $\mnr$.
\subsection{Basic Notation and Terminology}

For $A \in \mnc$, the transpose and conjugate transpose are denoted by $A^T$ and $A^*$, respectively.
The range, null space, and rank of $A$ are denoted by
\[
R(A), \quad N(A), \quad \mathrm{rk}(A) = \dim(R(A)).
\]

The diagonal matrix with entries $d_1, \ldots, d_n$ is written as $\diag(d_1, \ldots, d_n)$.

For index subsets $\alpha, \beta \subseteq \{1,2,\ldots,n\}$ (ordered increasingly),
$A[\alpha,\beta]$ denotes the submatrix of $A$ consisting of rows indexed by $\alpha$
and columns indexed by $\beta$. If either $\alpha$ or $\beta$ is empty, then
$A[\alpha,\beta]$ is vacuous. When $\alpha=\beta=\varnothing$, we define
$\det(A[\alpha,\beta])=1$. The submatrix $A[\alpha,\alpha]$, also denoted by
$A[\alpha]$, is called a principal submatrix of $A$.

Let $\alpha^c$ denote the complement of $\alpha$.
If $A[\alpha]$ is invertible, the Schur complement of $A[\alpha]$ in $A$ is defined by
\[
A / A[\alpha] := A[\alpha^c] - A[\alpha^c, \alpha]\, A[\alpha]^{-1} A[\alpha, \alpha^c].
\]
Then $A$ is invertible if and only if $A / A[\alpha]$ is invertible, in which case
\[
\det(A) = \det(A[\alpha]) \det(A / A[\alpha]).
\]
For further details on Schur complements and their applications, see \cite{Zhang2005}.

% For $B \in \mnc$, the spectrum and spectral radius are denoted by
% \[
% \sigma(B), \quad \rho(B) = \max\{ |\lambda| : \lambda \in \sigma(B) \}.
% \]

% \subsection{Matrix Exponential and Logarithm}

% For $X \in \mnr$, the matrix exponential is defined by the convergent power series
% \[
% e^{X} = \sum_{k=0}^\infty \frac{X^k}{k!}.
% \]
% Basic properties include:
% \begin{enumerate}[(i)]
%     \item $e^0 = I$;
%     \item $e^{X^T} = (e^X)^T$;
%     \item If $Y$ is invertible, then $e^{Y X Y^{-1}} = Y e^X Y^{-1}$;
%     \item If $X Y = Y X$, then $e^{X+Y} = e^X e^Y$.
% \end{enumerate}

% A real matrix $A$ admits a real logarithm (i.e., there exists $X \in \mnr$ such that $A = e^X$) if and only if \cite[Theorem 1]{culver}:
% \begin{enumerate}[(i)]
%     \item $A$ is invertible;
%     \item each Jordan block of $A$ corresponding to a negative eigenvalue occurs an even number of times.
% \end{enumerate}
% If $A$ has no negative real eigenvalues, the logarithm is unique and is called the \emph{principal logarithm}.
\subsection{Matrix Classes Under Consideration}

\begin{defn}{\rm
Let $A = (a_{ij}) \in \mnr$ or $\mnc$ as appropriate. 
The following classes of matrices are considered in this work:
\begin{itemize}
    \item 
    {\it Nonnegative} ({\it positive}) matrices, denoted by $A \ge 0$ $(A > 0)$, 
    if $a_{ij} \ge 0$ $(a_{ij} > 0)$ for all $i,j$. 
    Similarly, for a vector $x = (x_i) \in \mathbb{R}^n$, we write $x \ge 0$ $(x > 0)$ 
    if $x_i \ge 0$ $(x_i > 0)$ for all $i$.
\item A {\it signature matrix} is a diagonal matrix whose diagonal entries are $\pm1$.

    \item 
     {\it Hermitian} matrices are those satisfying $A = A^*$. 
    A Hermitian matrix $A$ is {\it positive semidefinite} (\PSD) if $x^* A x \ge 0$ for all $x$, 
    and {\it positive definite} (\PD) if $x^* A x > 0$ for all $x \ne 0$.

    \item 
    Matrices with {\it positive definite Hermitian part}:
    \[
    \Pi_n := \{ A \in \mnc : x^* H(A) x > 0 \text{ for all } x \ne 0 \}
    \]
    in which $H(A) = \tfrac{1}{2}(A + A^*)$ denotes the Hermitian part of $A$.

    \item 
    {\it Copositive} ({\it strictly copositive}) matrices: 
    a real symmetric matrix $A \in \mnr$ is copositive if $x^\top A x \ge 0$ for all $x \ge 0$, 
    and strictly copositive if $x^\top A x > 0$ for all nonzero $x \ge 0$.

   \item 
{\it Doubly nonnegative} (\DN) matrices are those that are both \PSD \ and entrywise nonnegative.

\item
A matrix is {\it totally nonnegative} (\TN) (resp.\ {\it totally positive} (\TP))
if all its minors are nonnegative (resp.\ positive).
    \item
A complex matrix is called a {\it $P$-matrix} if all of its principal minors are
positive numbers. 

    \item 
    A matrix $A \in \mnr$ is {\it semipositive} if there exists a vector $x \ge 0$ 
    such that $A x > 0$. 
    Any such $x$ is called a {\it semipositivity vector} of $A$.

    \item 
    A {\it $Z$-matrix} is a real matrix with nonpositive off-diagonal entries. 
    Every $Z$-matrix can be written as $A = sI - B$ in which $B \ge 0$. 
    If $s \ge \rho(B)$, then $A$ is an {\it $M$-matrix}; 
    if $s > \rho(B)$, then $A$ is a {\it nonsingular $M$-matrix}. 
    An invertible matrix $A$ is an {\it inverse $M$-matrix} (\IM) if $A^{-1}$ is an $M$-matrix.
    \item 
A real matrix $A$ is called {\it essentially nonnegative}  if $-A$ is a $Z$-matrix.
Equivalently, all off-diagonal entries of $A$ are nonnegative.

    \item 
    An {\it $H$-matrix} is a matrix $A = (a_{ij}) \in \mnc$ 
    whose {\it comparison matrix} $\mathcal{M}(A) = (m_{ij})$ defined by
    \[
    m_{ij} = 
    \begin{cases}
        |a_{ii}|, & i = j,\\
        -|a_{ij}|, & i \ne j,
    \end{cases}
    \]
    is an $M$-matrix. 
    The subclass of $H$-matrices for which $\mathcal{M}(A)$ is invertible 
    is denoted by $\mathcal{H}_I$.

\item 
An invertible matrix $A\in\mathbb{C}^{n\times n}$ is called an \emph{inverse $H$-matrix} if
$
A^{-1}\in \mathcal{H}_I.
$

    \item 
{A matrix $A$ is called {\it quasidominant} if there exists a positive vector 
$d=(d_i)$ such that
\[
d_i a_{ii} > \sum_{j \ne i} d_j |a_{ij}|, 
\quad \text{for all } i.
\]
Equivalently, $A$ is strictly generalized diagonally dominant with positive diagonal entries.}
\end{itemize}}
\end{defn}

For general background on matrix positivity and related classes, see \cite{berple, besh, HJ2, jst}. 

It is well known (see \cite[Fact 1.0.3.1]{HJ2}) that if $A \in \Pi_n$, then every eigenvalue $\lambda \in \sigma(A)$ satisfies $\Re(\lambda) > 0$. 
Moreover, from \cite[Lemma, p.~322]{london}, a matrix $A$ belongs to $\Pi_n$ if and only if it can be written as
\[
A = T \diag(1 + i\alpha_1, \ldots, 1 + i\alpha_n) T^*, 
\quad T \text{ nonsingular}, \quad \alpha_j \in \mathbb{R}.
\]

An important class closely related to $\Pi_n$ is that of $P$-matrices.
By \cite[Theorem~1]{moy}, a real matrix $A$ is a $P$-matrix if and only if
$SAS$ is semipositive for every signature matrix $S$.
In particular, every nonsingular $M$-matrix and its inverse are $P$-matrices, and hence semipositive. Let us just add the rather well-known fact that, if $A$ is a $Z$-matrix, 
then $A$ is an invertible $M$-matrix if and only if $A$ is semipositive. {Most of the real matrix positivity classes considered here are also semipositive.}

Recall that a $Z$-matrix $A$ is a nonsingular $M$-matrix precisely when it is inverse nonnegative, that is, $A^{-1} \ge 0$.
Further characterizations of inverse $M$-matrices can be given in terms of Schur complements 
(\cite[Theorems~5.9.7--5.9.8]{jst}). 
\begin{thm}\label{IMSchur1}
Let $A\geq0$. Then $A\in\IM$ if and only if $A$ has positive diagonal entries, all Schur complements are nonnegative, and all Schur complements of order $1$ are positive.
\end{thm}

The following result shows that it suffices to consider Schur complements of order $2$.

\begin{thm}\label{IMSchur2}
Let $A\geq0$. Then $A\in\IM$ if and only if $A$ has at least one positive diagonal entry, all Schur complements of order $2$ are nonnegative, and all Schur complements of order $1$ are positive.
\end{thm}

We conclude with a useful property of $H$-matrices. The reader is referred to \cite{bcgm,CS} for further details on $H$-matrices and inverse $H$-matrices.

\begin{rem}\label{signaturelysimilarh}
{\rm
If $A$ is an $H$-matrix and $S$ is a signature matrix, then
$\mathcal{M}(A)=\mathcal{M}(SAS)$, and consequently $SAS$ is also an
$H$-matrix.
A fundamental connection between quasidominant matrices, $H$-matrices,
and $P$-matrices was established in \cite[Theorem 2]{moy}: a square matrix
$A$ is quasidominant if and only if there exists $x>0$ such that
$SASx>0$ for every signature matrix $S$.
In particular, if $A$ is a real matrix with positive diagonal entries,
then $A$ is strictly generalized diagonally dominant if and only if it
belongs to the class $\mathcal{H}_I$.
It follows that any $A\in\mathcal{H}_I$ with positive diagonal entries
is necessarily a $P$-matrix.
}
\end{rem}

\subsection{Matrix Functions and Pick Functions}
To study both fractional and integer powers of matrices within key {positivity classes}, we employ classical concepts from the theory of matrix functions, particularly those related to {Pick functions}, which we briefly review here.

For a function $f$ analytic on a domain containing $\sigma(A)$, the matrix function $f[A]$ is defined by the Cauchy integral formula:
\[
f[A] = \frac{1}{2 \pi i} \int_\Gamma f(z) (z I - A)^{-1} \, dz,
\]
in which $\Gamma$ is a simple closed contour enclosing $\sigma(A)$.

This definition agrees with power series or polynomial substitution definitions (see \cite{hig}).

\begin{thm}[{\cite[Theorem 1.13]{hig}}]\label{promatrixfun}
{\rm For $A \in \mnc$ and $f$ analytic on $\sigma(A)$:
\begin{enumerate}[(a)]
    \item $f[A]$ commutes with $A$;
    \item if $X$ is invertible, then $f[X A X^{-1}] = X f[A] X^{-1}$;
    \item $\sigma(f[A]) = f(\sigma(A))$;
    \item if $A = \diag(A_{11}, \dots, A_{mm})$, then $f[A] = \diag(f[A_{11}], \dots, f[A_{mm}])$.
\end{enumerate}}
\end{thm}

\begin{thm}[{\cite[Theorems 1.15, 1.17]{hig}}]\label{combination}
{\rm Let $f, g$ be defined on $\sigma(A)$. Then:
\begin{enumerate}[(a)]
    \item If $h(t) = f(t) + g(t)$, then $h[A] = f[A] + g[A]$;
    \item If $h(t) = f(t) g(t)$, then $h[A] = f[A] g[A]$;
    \item If $h(t) = f(g(t))$ and both $g[A]$ and $f[g[A]]$ exist, then $h[A] = f[g[A]]$.
\end{enumerate}}
\end{thm}

\begin{defn}\label{pickdef}
{\rm Let
\(
D:=\mathbb{C}\setminus(-\infty,0].
\)
A complex function $f$ analytic on $D$ is called a \emph{Pick function} if
\[
\Im f(z)\,\Im z \ge 0, \quad \forall z\in D,
\]
in which $\Im z$ denotes the imaginary part of $z$.}
\end{defn}

Typical examples include $f(z) = z^t$ ($0 < t \leq 1$) and $f(z) = \log z$ (\cite{ando}).

Pick functions play a key role in matrix and operator theory (see \cite{ando1}).
A fundamental characterization is given by the following integral representation.

\begin{lem}[{\cite{don}}]\label{representation}
{\rm A function $f$ is a Pick function if and only if it can be represented as
\[
f(z) = a + b z + \int_0^\infty \frac{t z - 1}{t + z} \, d\mu(t)
\]
in which $a \in \mathbb{R}$, $b \geq 0$, and $\mu$ is a positive measure on $(0, \infty)$.}
\end{lem}

The representation is unique, with
\[
b = \lim_{y \to \infty} \frac{f(i y)}{i y}, \quad a = \Re f(i).
\]

\begin{rem}\label{matrixrepresentation}
{\rm For a Pick function $f$ and $A \in \mnc$ whose spectrum lies in $D$, combining Lemma \ref{representation} with Theorem \ref{combination} yields
\[
f[A] = a I + b A + \int_0^\infty \left[ t I - (1 + t^2)(A + t I)^{-1} \right] d\mu(t).
\]}
\end{rem}

\subsection{Matrix Exponential, Logarithm, and Arbitrary Powers}

For $X \in \mnr$, the matrix exponential is defined by the absolutely convergent power series
\[
e^{X} = \sum_{k=0}^\infty \frac{X^k}{k!}.
\]
Fundamental properties include:
\begin{enumerate}[(i)]
    \item $e^0 = I$;
    \item $e^{X^T} = (e^X)^T$;
    \item If $Y$ is invertible, then $e^{Y X Y^{-1}} = Y e^X Y^{-1}$;
    \item If $X Y = Y X$, then $e^{X+Y} = e^X e^Y$.
\end{enumerate}

A real matrix $A$ is said to have a \emph{real logarithm} if there exists $X \in \mnr$ such that $A = e^X$. 
By a classical result of Culver~\cite[Theorem~1]{culver}, this holds if and only if
\begin{enumerate}[(i)]
    \item $A$ is invertible; and
    \item every Jordan block of $A$ corresponding to a negative eigenvalue occurs an even number of times.
\end{enumerate}
If $A$ has no negative real eigenvalues, the logarithm is unique and called the \emph{principal logarithm}, denoted by $\log A$. 

\vspace{0.3em}
The existence and nature of matrix roots will be discussed through out the paper.   For details about matrix roots, we refer the reader to \cite{guohigham, hig,psa}.
		We recall at this stage, the following {essential} facts$\colon$
        
\noindent\textbf{Matrix Roots and Arbitrary Powers.}
A matrix $X \in \mnc$ is called a \emph{$p$th root} of $A$ if $X^p = A$, in which $p$ is a positive integer.  
When $A$ has no eigenvalues on the nonpositive real axis, the \emph{principal $p$th root} of $A$ is the unique $p$th root whose eigenvalues $\lambda$ satisfy 
\[
|\arg(\lambda)| < \frac{\pi}{p}.
\]

The principal root can be defined via the Cauchy integral representation
\[
A^{1/p} = \frac{1}{2 \pi i} \int_{\Gamma} z^{1/p} (z I - A)^{-1} \, dz,
\]
in which $\Gamma$ is a closed contour enclosing $\sigma(A)$ and lying in the domain of analyticity of $z^{1/p}$.

\begin{thm}[{\cite[Theorem 5]{hig1}}]\label{existrootnegativeeig_re}
{\rm Let $A$ be an invertible real matrix. Then roots of all integer orders of $A$ exist if and only if every Jordan block of $A$ corresponding to a negative eigenvalue occurs an even number of times.}
\end{thm}

If $A$ admits a real logarithm $X$, then, for a fixed choice of $X$, 
the real powers of $A$ associated with $X$ are defined by
\[
A^t:=e^{tX},\qquad t\in\mathbb{R}.
\]
If, in addition, the spectrum of $A$ is contained in 
$\mathbb{C}\setminus(-\infty,0]$ and $X=\log A$ is the principal 
logarithm, then this definition agrees with the analytic functional 
calculus for the principal branch of $f(z)=z^t$. In particular, for 
$0\leq t\leq 1$, the function $f(z)=z^t$ is a Pick function 
(see Definition~\ref{pickdef}).

{
In Section~3, the relevant spectral conditions and assumptions on the matrix classes ensure the existence of the principal powers under consideration. Thus, real matrix powers $A^t$, including roots $A^{1/p}$, are understood to be principal, except in the discussion of strongly infinitely divisible matrices (\SIDM), where roots are considered in the sense specified there.

The study of integer powers and matrix roots is naturally connected to that of
fractional powers. If a matrix class is closed under positive integer powers and the
relevant matrix roots, then it is also closed under positive rational powers, since
$
A^{k/p}=\left(A^{1/p}\right)^k,\quad k,p\in\mathbb{N}.
$
Under suitable continuity and closedness assumptions, this preservation may further
extend from positive rational powers to arbitrary positive real powers.
}

Thus, by Culver's theorem and Theorem~\ref{existrootnegativeeig_re}, an invertible real matrix has roots of all positive integer orders if and only if it has a real logarithm. This immediately yields the following characterization.
\begin{thm}\label{arbroots_re}
Let $A \in \mnr$ be invertible. Then the following are equivalent:
\begin{enumerate}[(i)]
    \item Roots of all positive integer orders of $A$ exist;
    \item There exists $B \in \mnr$ such that $A^t = e^{t B}$ for all $t \geq 0$.
\end{enumerate}
\end{thm}

Thus, for matrices with spectra contained in $\mathbb{C} \setminus (-\infty,0]$, the exponential, logarithm, and all fractional powers are well defined and consistent under the matrix functional calculus.

\section{Powers in Matrix Positivity Classes}
\label{Powers}

In this section, we investigate whether the properties and structures defining various matrix positivity classes are preserved under taking powers. Here, the term \emph{power} includes both integer and fractional powers of matrices. Integer powers are uniquely defined. In contrast, an $m$th root of a matrix, viewed as a solution of
$
X^m=A,
$
need not be unique. However, if
$
\sigma(A)\cap(-\infty,0]=\varnothing,
$
then the principal fractional powers (and hence the corresponding principal roots) are uniquely determined via the functional calculus. Throughout this section, powers and roots are understood in this sense.

\subsection{Fractional Powers}

We now investigate the preservation of matrix positivity classes under Pick functions. 
Let $\mathcal{C}$ be one of the matrix positivity classes under consideration. 
We study conditions under which a Pick function $f$ preserves $\mathcal{C}$; that is, 
whenever $A \in \mathcal{C}$, when is $f[A] \in \mathcal{C}$?
Since fractional powers are associated with the Pick function
$
f(z)=z^p,\; 0<p<1,
$
this framework naturally yields preservation results for fractional powers $A^p$. 
In particular, whenever the relevant Pick functions preserve $\mathcal{C}$, 
the principal fractional power $A^p$ belongs to $\mathcal{C}$ for $0<p<1$.

\subsubsection{Positive Stable and Positive Definite}

We begin by recalling a related result for \(P\)-matrices, and then develop analogous results for positive stable, positive definite, and Hermitian part positive definite matrices. In particular, we show that fractional powers of such matrices remain within the same class.
\begin{thm}\label{realeignofpmatrix}\rm{\cite[Theorem 3.9]{SST}}
		Let $A$ be a $P$-matrix, and $f$ be a Pick function with $f(t)>0$ for all $t>0.$ Then the real eigenvalues of $f[A]$ are positive.
\end{thm}
% \begin{thm}\label{pmatrixpreserve}{\rm\cite[Theorem 3.14]{SST}}
% 	Let $A\in \mnc $ be a $P$-matrix, and let $f$ be a Pick function with $f(t)>0$ 
% 	  	for all $t>0.$ Then $f[A]$ is a $P$-matrix. 
% \end{thm}

%  \begin{cor}\label{fractionalpowerofp}
% 	If $A\in \mnc$ is a $P$-matrix, then $A^t$ is a $P$-matrix for $0< t\leq 1.$ 
% 	In particular, all roots of a $P$-matrix are $P$-matrices.
% \end{cor}

Using a similar technique as the proof of Theorem \ref{realeignofpmatrix}, we can prove the following theorem:
\begin{thm}\label{pickpositivestable}
   Let $A$ be positive stable, and $f$ be a Pick function with $f(t)>0$ for all $t>0.$ Then $f[A]$ is positive stable. 
\end{thm}
\begin{proof}
    {Since every real eigenvalue of a positive stable matrix $A$ is positive, 
the matrix function $f[A]$ is defined using the standard functional calculus 
for any Pick function $f$.}
		Let us use the representation of $f$ from Lemma \ref{representation} . Then,
		\begin{equation*}
			f[A]=aI+bA+ \int_{0}^{\infty} \left[tI-(1+t^2)(A+tI)^{-1}\right] \,d\mu.
		\end{equation*}
		Since $f(t)>0$ for all $t>0,$ $f(0) \geq 0$, we have
		\begin{equation*}
			f(0)=a+ \int_{0}^{\infty} \left[t-(1+t^2)t^{-1}\right] \,d\mu.
		\end{equation*}
		Substituting $a$ in the expression of Lemma \ref{representation}, we get 
		 \begin{eqnarray*}\label{conditionpick}
      f(z) & = & {f(0)+bz+ \int_{0}^{\infty} (1+t^2) \left[ \frac{1}{t}- \frac{1}{z+t}\right]\,d\mu}\\
       &=& f(0)+bz+ \int_{0}^{\infty} \frac{z(1+t^2)}{t(z+t)} \,d\mu.
   \end{eqnarray*}
   {It follows that $f(z)$ has positive real part whenever $z$ has positive real part. Since each  eigenvalue of $A$ has positive real part, by the spectral mapping theorem, we can now conclude that each eigenvalue of $f[A]$ has positive real part}. Thus, $f[A]$ is positive stable. 
   \end{proof} 

  \begin{cor}
If $A\in\mnc$ is positive stable, then $A^t$ is positive stable for $0<t\leq1$.
In particular, all principal roots of a positive stable matrix are positive stable matrices.
\end{cor}

Using the spectral decomposition, it is well known that fractional powers of positive definite matrices are again positive definite. To extend this preservation property from fractional powers to general Pick functions, we first establish that Pick functions preserve Hermitian symmetry. While \(f[A^T]=f[A]^T\) always holds \cite[Theorem 1.13(b)]{hig}, the identity \(f[A^*]=f[A]^*\) is not valid for arbitrary analytic functions. The following theorem characterizes precisely when the latter property holds.

\begin{thm}{\rm \cite[Theorem 1.13(b)]{hig}}\label{conjugateofmatrix}
    {\rm Let $f$ be an analytic function defined on an open subset $\Omega \subseteq \mathbb{C}$ such that each connected component of $\Omega$ is closed under conjugation. Consider the corresponding matrix function $f$ defined on its natural domain in $\mathbb{C}^{n \times n}$, denoted by $\mathcal{D}=\{A\in \mathbb{C}^{n \times n} \colon \sigma(A) \subseteq \Omega \}$. Then the following statements are equivalent$\colon$
    
    (a)\, $f[A^*]=f[A]^*$ for all $A\in \mathcal{D}.$
    
    (b)\, $f(\mathbb{R}\cap \Omega)\subseteq \mathbb{R}.$}
    \end{thm}
 \begin{obs}\label{Hermitian}
     {\rm Assuming $f$ is a Pick function, Lemma \ref{pickdef} implies that $f(z)$ is real for real $z$. Consequently, $f(\mathbb{R}\cap D) \subseteq \mathbb{R}$ in which $D$ is defined in Definition \ref{pickdef}. By applying Theorem \ref{conjugateofmatrix}, we conclude that $f[A^*]=f[A]^*$. Therefore, if $A$ is Hermitian, then $f[A]$ is also Hermitian.}
 \end{obs}

\begin{thm}\label{positivedefinite}
    {\rm Let $A$ be a positive definite matrix, and $f$ be a Pick function with $f(t)>0$ for all $t>0.$ Then  $f[A]$ is a positive definite matrix.} 
\end{thm}
\begin{proof}
 Observation \ref{Hermitian}, implies that if $A$ is Hermitian, then $f[A]$ is also Hermitian. Thus, the statement follows directly from Theorem \ref{pickpositivestable}.
\end{proof}
\begin{cor}\label{PDroots}
{\rm If $A\in\mnc$ is a positive definite matrix, then $A^t$ is a positive definite matrix for $0<t\leq1$.
In particular, all principal roots of a positive definite matrix are positive definite matrices.}
\end{cor}

We now turn our attention to roots of matrices in \(\Pi_n\).

\begin{thm}\label{pickharmitianpart}
   {\rm  Let $A\in \Pi_n$, and $f$ be a Pick function with $f(t)>0$ for all $t>0.$ Then  $f[A]\in \Pi_n.$ }
\end{thm}
\begin{proof}
  % Considering the representation of $f$ from Lemma \ref{representation}, we note that $f(t)$ is real when $t$ is real. The condition ${\rm Re}f(z)>0$ for all $z$ such that ${\rm Re}(z)>0$ ensures that $f(t)>0$ for all $t>0$. 
 Since \( A \in \Pi_n \), we have \( H(A) \succ 0 \), so \( \text{Re}(x^* A x) > 0 \) for all \( x \ne 0 \). It follows that the spectrum \( \sigma(A) \subset \mathbb{C}_+ \), and hence \( f[A] \) is well-defined via the holomorphic functional calculus.
 
By employing the representation of the Pick function under the condition $f(t)>0$ for all $t>0$, as provided in Theorem \ref{pickpositivestable}, which states$\colon$
  \begin{equation*}
      f(z)= f(0)+bz+ \int_{0}^{\infty} \frac{z(1+t^2)}{t(z+t)} \,d\mu
  \end{equation*}
 in which $f(0)\geq 0.$

We evaluate the Hermitian part of \( f[A] \) by analyzing the real part of the quadratic form:
\[
x^* f[A] x = f(0)\|x\|^2 + b x^* A x + \int_0^\infty x^* \left( \frac{A(1 + t^2)}{t(A + tI)} \right) x \, d\mu(t),
\]
which gives
\[
\text{Re}(x^* f[A] x) = f(0)\|x\|^2 + b\,\text{Re}(x^* A x) + \int_0^\infty \text{Re}\left( x^* \left( \frac{A(1 + t^2)}{t(A + tI)} \right) x \right) d\mu(t).
\]

We analyze the integrand more closely. Define
\[
M_t := \frac{A(1 + t^2)}{t(A + tI)} = (1 + t^2)\left( \frac{1}{t}I - (A + tI)^{-1} \right),
\]
so that
\[
x^* M_t x = (1 + t^2)\left( \frac{1}{t}\|x\|^2 - x^*(A + tI)^{-1}x \right),
\]
and thus
\[
\text{Re}(x^* M_t x) = (1 + t^2)\left( \frac{1}{t}\|x\|^2 - \text{Re}(x^*(A + tI)^{-1}x) \right).
\]

Now observe that since \( A \in \Pi_n \), the Hermitian part \( H(A) \succ 0 \), so all eigenvalues of \( A \) lie in \( \mathbb{C}_+ \). Therefore, for any \( t > 0 \), the matrix \( A + tI \) also has spectrum in \( \mathbb{C}_+ \), and is thus invertible. Moreover, each eigenvalue \( \lambda \in \sigma(A) \subset \mathbb{C}_+ \) satisfies \( \text{Re}(\lambda + t) > 0 \), so
\[
\sigma((A + tI)^{-1}) = \left\{ \frac{1}{\lambda + t} : \lambda \in \sigma(A) \right\} \subset \mathbb{C}_+.
\]

Hence, the matrix \( (A + tI)^{-1} \) has positive definite Hermitian part, i.e., \( (A + tI)^{-1} \in \Pi_n \), and so for all \( x \ne 0 \),
\[
\text{Re}(x^* (A + tI)^{-1} x) > 0.
\]

On the other hand, the scalar \( \frac{1}{t} \|x\|^2 \) is an upper bound for \( \text{Re}(x^* (A + tI)^{-1} x) \), since \( A + tI \succ tI \Rightarrow (A + tI)^{-1} \prec \frac{1}{t}I \). Therefore,
\[
\frac{1}{t} \|x\|^2 - \text{Re}(x^* (A + tI)^{-1} x) > 0 \quad \text{for all } x \ne 0,
\]
and thus
\[
\text{Re}(x^* M_t x) > 0.
\]

This shows that the integrand in the expression for \( \text{Re}(x^* f[A] x) \) is strictly positive for all \( x \ne 0 \) and all \( t > 0 \). Since the measure \( \mu \) is positive, the integral is strictly positive as well.

Finally, combining all terms, we have
\[
\text{Re}(x^* f[A] x) > 0 \quad \text{for all } x \in \mathbb{C}^n \setminus \{0\},
\]
which implies that \( H(f[A]) = \frac{f[A] + f[A]^*}{2} \succ 0 \), i.e., \( f[A] \in \Pi_n \).

\end{proof}

\begin{cor}
{\rm If $A\in\Pi_n$, then $A^t\in\Pi_n$ for $0<t\leq1$.
In particular, all principal roots of a matrix in $\Pi_n$ are in $\Pi_n$.}
\end{cor}

% For a \( 2 \times 2 \) copositive matrix  
% \[
% A = 
% \begin{pmatrix}
% a & b \\
% b & c
% \end{pmatrix},
% \]
% copositivity implies that \( a, c \ge 0 \) and \( b \ge -\sqrt{ac} \), ensuring that for every nonnegative vector \( x \), the quadratic form \( x^\top A x \) is nonnegative.  

% Because \( A \) is symmetric, its eigenvalues \( \lambda_1, \lambda_2 \) are real and satisfy  
% \[
% \lambda^2 - (a + c)\lambda + (ac - b^2) = 0.
% \]
% Given \( b \ge -\sqrt{ac} \), we have \( ac - b^2 \ge 0 \), so both eigenvalues are nonnegative. Hence, for \( 2 \times 2 \) copositive matrices, the eigenvalues are always nonnegative, and such matrices are therefore positive semidefinite.  

% However, for dimensions \( n \ge 3 \), copositive matrices can have negative eigenvalues and hence are \textbf{not necessarily positive semidefinite}. For example,  
% \[
% B =
% \begin{pmatrix}
% 1 & -1 & 1 \\
% -1 & 1 & 1 \\
% 1 & 1 & 1
% \end{pmatrix},
% \]
% is copositive because, for any nonnegative vector \( x = (x_1, x_2, x_3)^\top \ge 0 \),  
% \[
% x^\top B x = (x_1 - x_2)^2 + 2x_1x_3 + 2x_2x_3 + x_3^2 \ge 0.
% \]
% Yet, \( B \) has eigenvalues \( -1, 2, 2 \), including a negative eigenvalue. Thus, while \( B \) is copositive, it is not positive semidefinite.  

% This example highlights that copositivity is a weaker condition than positive semidefiniteness, particularly in higher dimensions.  

% 

We conclude by considering copositive matrices, whose spectral characteristics differ from those of positive definite matrices.
Copositive matrices may have negative eigenvalues. For example,
\[
B = \begin{pmatrix} 1 & -1 & 1 \\ -1 & 1 & 1 \\ 1 & 1 & 1 \end{pmatrix}
\]
is copositive, yet its eigenvalues are \(-1, 2, 2\), showing that copositive matrices need not be positive semidefinite.

% \begin{obs}
% {\rm Consider a \emph{strictly copositive} matrix \( A \) for which the functional calculus \( f[A] \) of a Pick function \( f \) is defined. Since Pick functions are analytic on the complex plane excluding the nonpositive real axis, the eigenvalues of \( A \) must be positive for \( f[A] \) to be well-defined. This implies that \( A \) must be positive definite. Therefore, Pick functions preserve strict copositivity \emph{if and only if} \( A \) is positive definite, and in this case, \( f[A] \) remains strictly copositive. In particular, fractional powers of a copositive matrix are copositive if and only if the matrix is positive definite.}
% \end{obs}

\begin{obs}
{\rm
Consider a \emph{strictly copositive} matrix \(A\) for which the functional calculus
\(f[A]\) of a Pick function \(f\) is defined. Since Pick functions are analytic on the
complex plane excluding the nonpositive real axis, the eigenvalues of \(A\) must be
positive for \(f[A]\) to be well-defined. This implies that \(A\) must be positive
definite. Therefore, a Pick function \(f\) can preserve strict copositivity of \(A\),
that is, \(f[A]\) is strictly copositive, only when \(A\) is positive definite.
In this case, \(f[A]\) is positive definite and hence strictly copositive.
In particular, the principal fractional powers \(A^p\), \(0<p<1\), are strictly
copositive whenever \(A\) is positive definite.
}
\end{obs}

\subsubsection{$M$-matrices, Inverse $M$-matrices, and Semipositive Matrices}

We now turn to fractional powers of \(M\)-matrices, inverse \(M\)-matrices, and semipositive matrices. The preservation of \(M\)-matrices and their inverses under matrix functions has been studied extensively. In particular, it was shown in \cite{ando} that Pick functions preserve the class of \(M\)-matrices under suitable conditions. Furthermore, \cite{bapcatneu} characterized those functions that preserve inverse \(M\)-matrices or map them into \(M\)-matrices.

We now summarize some key results related to Pick functions and their action on these matrix classes.

\begin{thm}
		{\rm \cite[Theorem 3.1]{ando}} \label{pickM} {\rm Let $f$ be a Pick function and $A$ be an invertible $M$-matrix. Then $f[A]$ is a $Z$-matrix. If, in addition, $f(t)>0$ for all $t>0$, then $f[A]$ is an invertible $M$-matrix.}
	\end{thm}

	A statement, verbatim to Theorem \ref{pickM} holds for the class of inverse $M$-matrices, too.  Let $A$ be an inverse $M$-matrix, and let $f$ be a Pick function with $f(t)>0$ for all $t>0$. Then, $f[A]$ is an inverse $M$-matrix (see para after Theorem $3.2$, \cite{bapcatneu}).

	Applying Theorem \ref{pickM} to $f(z)=z^p$ for $0<p\leq 1,$ the following consequences are obtained.
 
	\begin{thm}\label{fram}
{\rm \cite[Theorem~3.2]{ando}}
{\rm If \(A\) is an invertible \(M\)-matrix, then \(A^p\) is an invertible \(M\)-matrix for \(0<p\le 1\). In particular, all principal roots of an invertible \(M\)-matrix are invertible \(M\)-matrices.}
\end{thm}

\begin{thm}
{\rm If \(A\) is an inverse \(M\)-matrix, then \(A^p\) is an inverse \(M\)-matrix for \(0<p\le 1\). In particular, all principal roots of an inverse \(M\)-matrix are inverse \(M\)-matrices.}
\end{thm}

% This is not true, as demonstrated by the following example. 
% \begin{thm}
%    Let $A$ be semipositive so that the spectrum of $A$ does not contain nonpositive real numbers. Then $A^p$ is semipositive for $0<p\leq 1.$ 
% \end{thm}
%    \begin{proof}
%      Let $f(z)=z^p,$ where $0<p\leq 1,$ be a Pick function.  Since the spectrum of $A$ does not contain nonpositive real numbers, the matrix $f[A]$ is well-defined. Now, from Equation \ref{matrixfunxtion}, we have 
%      \begin{eqnarray*}
%        f[A]=f(0)I+bA+\int_{0}^{\infty} \frac{(1+t^2)A}{t(A+tI)} \,d\mu,
%    \end{eqnarray*} 
%    where $f(0)=0$ and $b=lim_{y\rightarrow \infty}\frac{f(iy)}{y}=0.$
    
%      Therefor, we conclude that
%      \begin{eqnarray*}
%        f[A]=\int_{0}^{\infty} \frac{(1+t^2)A}{t(A+tI)} \,d\mu.
%    \end{eqnarray*}
% Let $x$ is a semipositive vector of $A.$ Now, suppose $Ax=y.$ For $t>0,$ the vector $y+tx$ is a semipositive vector of $\frac{(1+t^2)A}{t(A+tI)}.$ 

% Thus, $f[A]=A^p,$ for $0<p\leq 1,$ is semipositive. \TR{How can we conclde this sentence?}
%  \end{proof} 
% \TB{I went through a few numerical examples that show that the above theorem seems to be correct.}

% \section*{Semipositivity of Fractional Powers of Matrices}
\noindent These results motivate the following theorem that extends the conclusion to all semipositive matrices whose spectra avoid the nonpositive real axis.

\begin{thm}\label{rootsofsemipositive}
{\rm Let $A$ be a semipositive matrix whose spectrum contains no nonpositive real numbers. 
Then, for any $0<p\leq 1$, the principal matrix power $A^p$ is semipositive.}
\end{thm}
\begin{proof}
Let $ f(z) = z^p $ for $ 0 < p \leq 1 $, which is a Pick function. Since the spectrum of $ A $ does not contain nonpositive real numbers, the matrix function $ f[A] $ is well-defined via the integral representation
$$
f[A] = f(0) I + b A + \int_0^\infty \frac{(1 + t^2) A}{t (A + t I)} \, d\mu(t)
$$
in which $ f(0) = 0 $,  and $$
b = \lim_{y \to \infty} \frac{f(iy)}{y} = \lim_{y \to \infty} \frac{(iy)^p}{y} = 0,
$$
since $ 0 < p < 1 $. 
Thus,
$$
A^p = f[A] = \int_0^\infty \frac{(1 + t^2)}{t} A (A + t I)^{-1} \, d\mu(t).
$$

Let $ x > 0 $ be a semipositive vector for $ A $, so that
$
A x = y > 0.
$
Define
$$
v_t := (A + t I) x = y + t x > 0, \quad \text{for all } t > 0.
$$
Since $ v_t = (A + t I) x $, it follows by direct algebraic manipulation that
$$
(A + t I)^{-1} v_t = x > 0.
$$
Applying the operator inside the integral to $ v_t $, we get
$$
K(t) v_t := \frac{(1 + t^2)}{t} A (A + t I)^{-1} v_t = \frac{(1 + t^2)}{t} A x = \frac{(1 + t^2)}{t} y > 0.
$$

That is, for each $ t > 0 $, the operator $ K(t) $ maps a positive vector to a positive vector, so each $ K(t) $ is semipositive.

Now fix any $ t_0 > 0 $ and define:
$
z := v_{t_0} = A x + t_0 x > 0.
$
Then:
$
K(t_0) z > 0.
$
Moreover, since $ K(t) $ is continuous in $ t $, it follows that $ K(t) z > 0 $ for all $ t $ in a neighborhood of $ t_0 $. Because the measure $ \mu $ is a positive Borel measure and the integrand $ K(t) z $ is strictly positive on a set of positive $ \mu $-measure, we conclude that
$$
A^p z = \int_0^\infty K(t) z \, d\mu(t) > 0.
$$
Thus, $ A^p $ maps a positive vector $ z $ to a positive vector, and hence is semipositive.
\end{proof}

\begin{rem}
{\rm The spectral assumption in Theorem~\ref{rootsofsemipositive} is necessary to ensure that $f[A]$ is well-defined. This is because a semipositive matrix can possess negative eigenvalues, and the functional calculus for Pick functions is only valid when the spectrum of $A$ lies outside the nonpositive real axis.}
\end{rem}

% \begin{thm}
% 		Let $A$ be semipositive with $A$ and $A^2$ have same semipositivity vector, and $f$ be a Pick function with $f(t)>0$ for all $t>0$ such that $f[A]$ is defined.  Then  $f[A]$ is semipositive.
% \end{thm}
% \begin{proof}
% % {Since every real eigenvalue of a $P$-matrix $A$ is positive, 
% % the matrix function $f[A]$ is defined using the standard functional calculus 
% % for any Pick function $f$.}
% By employing the representation of the Pick function under the condition $f(t)>0$ for all $t>0$, as provided in Theorem \ref{pickpositivestable}, which states$\colon$
%   \begin{equation*}
%       f(z)= f(0)+bz+ \int_{0}^{\infty} \frac{z(1+t^2)}{t(z+t)} \,d\mu,
%   \end{equation*}
%  where $f(0)\geq 0.$
%    Now, \begin{eqnarray}\label{matrixfunxtion}
%        f[A]=f(0)I+bA+\int_{0}^{\infty} \frac{(1+t^2)A}{t(A+tI)} \,d\mu.
%    \end{eqnarray}
 
%    % Since $f(0), \; b \geq 0$ and $Ax>0,$ $f[A]x>0.$ Therefore $f[A]$ is semipositive. 

%    Given that $A$ and $A^2$ share the same semipositive vector, say $x$, we have that $Ax$ and $A^2x$ are both positive vectors. Now, suppose $Ax=y.$ For $t>0,$ the vector $y+tx$ is a semipositive vector for $(A+tI)^{-1}.$ It is straightforward to verify that $y+tx$ is also a semipositive vector for $\frac{(1+t^2)A}{t(A+tI)}$ and $A.$
   
%    Using the argument above we can assure that $f[A]$ is semipositive. 
%    \end{proof} 

\begin{thm}
{\rm Let $ A $ be semipositive, and suppose that $ A $ and $ A^2 $ share the same semipositive vector. Let $ f $ be a Pick function such that $ f(t) > 0 $ for all $ t > 0 $, and assume that $ f[A] $ is defined. Then $ f[A] $ is semipositive.}
\end{thm}

\begin{proof}
By employing the representation of the Pick function under the condition $ f(t) > 0 $ for all $ t > 0 $, as provided in Theorem~\ref{pickpositivestable}, we have:
$$
f(z) = f(0) + bz + \int_0^{\infty} \frac{z(1 + t^2)}{t(z + t)} \, d\mu
$$
in which $ f(0) \geq 0 $.
This gives the matrix representation
\begin{equation} \label{matrixfunction}
f[A] = f(0) I + b A + \int_0^{\infty} \frac{(1 + t^2) A}{t(A + t I)} \, d\mu.
\end{equation}

Given that $ A $ and $ A^2 $ share the same semipositive vector, say $ x $, we have that $ A x = y > 0 $ and $ A^2 x > 0 $.
Now, for any $ t > 0 $, the vector
$
v_t := y + t x = A x + t x = (A + t I)x > 0
$
is a semipositive vector for $ (A + t I)^{-1} $. It is straightforward to verify that $ v_t $ is also a semipositive vector for the operator
$$
\frac{(1 + t^2) A}{t(A + t I)},
$$
and for $ A $. Using the argument above, we conclude that $ f[A] $ is semipositive.
\end{proof}

% \begin{thm}
%   Let $A$ be semimonotone so that the spectrum of $A$ does not contain nonpositive real numbers. Then $A^p$ is semimonotone for $0<p\leq 1.$   
% \end{thm}
\subsubsection{$H$-matrices and Inverse $H$-matrices}

It was shown in \cite{luguo2018} that roots of \(H\)-matrices with positive diagonal entries remain \(H\)-matrices. However, it is not known whether this preservation property extends to general Pick functions. We investigate this question for both \(H\)-matrices and inverse \(H\)-matrices. Since an \(H\)-matrix may have negative eigenvalues, we restrict our attention to matrices in \(\mathcal{H}_I\), that is, \(H\)-matrices with positive diagonal entries. Under this assumption, all real eigenvalues are positive, ensuring that the matrix function \(f[A]\) is well defined for the Pick functions considered below. An analogous assumption is imposed on inverse \(H\)-matrices.

Let us recall the well-known theorem of Ostrowski \cite{ostro}: for 
$A\in\mathcal{H}_I$ (the class of invertible $H$-matrices),
\begin{equation}
    |A^{-1}| \leq \mathcal{M}(A)^{-1}.
\end{equation}
	
	\begin{thm}\label{inequalityofh}
		{\rm Let $A\in \mnr$. If $A\in \mathcal{H}_I$ with positive diagonal entries and also let $f$ be a Pick function. Then 
		\begin{equation*}
			f[\mathcal{M}(A)] \leq f[A]. 
		\end{equation*} }
	\end{thm}
	\begin{proof}
Given that $A$ is real and $A\in \mathcal{H}_I$ with positive diagonal entries, every real eigenvalue of $A$ is positive. Consequently, the matrix function $f[A]$ is defined using the standard functional calculus for any Pick function $f$.
		Let us use the representation of $f$ from Lemma \ref{representation}. Then
		\begin{equation*}
			f[A]=aI+bA+ \int_{0}^{\infty} \{tI-(1+t^2)(A+tI)^{-1}\} \,d\mu.
		\end{equation*}
		Since $a_{ii}>0$ for all $i,$
		\begin{equation*}
			\mathcal{M}(A+tI)=\mathcal{M}(A)+tI \\\ for \\\ all \\\ t>0.
		\end{equation*}
		Hence $A+tI \in \mathcal{I}_{H}$, for all  $t>0.$
		Therefore by Ostrowski's Theorem we get
		\begin{equation*}
			|(A+tI)^{-1}| \leq (\mathcal{M}(A+tI))^{-1}=(\mathcal{M}(A)+tI)^{-1}.
		\end{equation*}
		Therefore
		\begin{equation*}
			\int \{tI-(1+t^2)(\mathcal{M}(A)+tI)^{-1}\} \, d\mu \leq  \int \{tI-(1+t^2)(A+tI)^{-1}\} \, d\mu.
		\end{equation*}
		Again, since $b \geq0,$ 
		\begin{equation*}
			b \mathcal{M}(A) \leq bA.  
		\end{equation*}
		Now, it follows that $f[\mathcal{M}(A)] \leq f[A],$ completing the proof.
		
		% \TB{$f[\mathcal{M}(A)] \leq f(|A|)\leq f(A).$}
	\end{proof}
	
	\begin{obs}\label{h}
{\rm
Let $A\in \mathcal{H}_{I}\cap \mnr$ with positive diagonal entries, and let $f$ be a Pick function. Then the matrix $f[A]$ is real. If, in addition, $f(t)>0$ for all $t>0$, then $f[A]$ is a real matrix with positive diagonal entries, since $f[\mathcal{M}(A)]$ is an invertible $M$-matrix.}
\end{obs}
	\begin{thm}\label{pickh}
{\rm Let $A\in \mnr$. If $A\in \mathcal{H}_I$ has positive diagonal entries and $f$ is a Pick function such that $f(t)>0$ for all $t>0$, then $f[A]\in \mathcal{H}_I$ with positive diagonal entries.}
\end{thm}
	\begin{proof}
 Let $S$ be any signature matrix. Considering Remark \ref{signaturelysimilarh} and Observation \ref{h}, we conclude that $SAS$ is a real matrix with positive diagonal entries, belonging to $\mathcal{H}_I$. Therefore, the matrix function $f[SAS]$ is well-defined for any Pick function.
	Now, according to Theorem \ref{inequalityofh} and the representation of $f$ (Lemma \ref{representation}), we have:
	 $$f[\mathcal{M}(A)]=f[\mathcal{M}(SAS)]\leq f[SAS]=Sf[A]S.$$
Since $\mathcal{M}(A)$ is an invertible $M$-matrix, by Theorem \ref{pickM}, $f[\mathcal{M}(A)]$ is also an invertible $M$-matrix. Thus, there exists $x>0$ such that $0<f[\mathcal{M}(A)]x\leq Sf[A]Sx$. Therefore, by \cite[Theorem 2]{moy}, $f[A]\in \mathcal{H}_I$ with positive diagonal entries.
	\end{proof}

\begin{cor}
{\rm Let $A \in \mnr$. If $A$ is strictly generalized diagonally dominant with positive diagonal entries, and $f$ is a Pick function such that $f(t) > 0$ for all $t > 0$, then $f[A] \in \mathcal{H}_I$ with positive diagonal entries.  }    
\end{cor}

\begin{proof}
Since $A$ is strictly generalized diagonally dominant with positive diagonal entries, its comparison matrix $\mathcal{M}(A)$ is a strictly generalized diagonally dominant $Z$-matrix. Hence $\mathcal{M}(A)$ is an invertible $M$-matrix, and therefore $A\in\mathcal{H}_I$. The result now follows from Theorem~\ref{pickh}.
\end{proof}

We now consider the corresponding class of inverse $H$-matrices,
\[
\mathcal{IH}_{+}
=
\left\{
A\in\mnr :
A^{-1}\in\mathcal H_I
\text{ and } A^{-1} \text{ has positive diagonal entries}
\right\}.
\]
\begin{thm}\label{pickih}
{\rm Let $A\in\mathcal{IH}_{+}$, and let $f$ be a Pick function such that
$f(t)>0$ for all $t>0$. Then $f[A]\in\mathcal{IH}_{+}$.}
\end{thm}

\begin{proof}
    Since $f$ is a Pick function with $f(t)>0$ for all $t>0$, the function $g(z)=f(z^{-1})^{-1}$ is again a Pick function \cite{ando}.  Consequently, the theorem follows from Theorem \ref{pickh}.
\end{proof}
Applying Theorem \ref{pickh} and \ref{pickih} to $f(z)=z^p$ for $0<p\leq 1,$ yields the following consequences respectively. 
\begin{cor}
	{\rm 	If $A\in \mathcal{H}_I$ has positive diagonal entries, then $A^p\in \mathcal{H}_I$ with positive diagonal entries for $0<p\leq 1.$ In particular, all roots of a matrix $A$ in  $\mathcal{H}_I$ that has positive diagonal entries  belong to the same matrix classes with positive diagonal entries.}
	\end{cor}
   \begin{cor}
{\rm Let $A \in \mnr$. If $A$ is strictly generalized diagonally dominant with positive diagonal entries, then $A^p$ is also strictly generalized diagonally dominant with positive diagonal entries for $0 < p \leq 1$. In particular, all roots of a strictly generalized diagonally dominant matrix with positive diagonal entries belong to the same matrix class and have positive diagonal entries.}
\end{cor}

 \begin{cor}\label{frah}
		{\rm If $A\in\mathcal{IH}_{+}$, then $A^p\in \mathcal{IH}_{+}$ for $0<p\leq 1.$ In particular, all roots of a matrix $A$ in $\mathcal{IH}_{+}$  belong to the same matrix classes.}
	\end{cor}
\subsubsection{Strongly Infinitely Divisible Matrices}
In the context of inverse $M$-matrices (\IM), it was established in \cite{john} that matrices in \IM\ have arbitrary roots within \IM. Although \IM\ is not closed under taking powers, every power and every root of an element of \IM\ is a nonnegative matrix with positive diagonal entries. This naturally leads to the question of whether there exist invertible nonnegative matrices outside \IM\ that share these properties.

To study this question, we recall the notions of infinitely divisible and strongly infinitely divisible matrices.

\begin{defn}
{\rm
A nonnegative matrix $A\in\mnr$ is said to be an {\it infinitely divisible matrix}
(\IDM) if, for every positive integer $m$, there exists a nonnegative matrix $K_m$ such that
\[
(K_m)^m=A.
\]

If, in addition, $A$ is invertible, then $A$ is called a {\it strongly infinitely divisible matrix}
(\SIDM).
}
\end{defn}

Van-Brunt \cite{van} answered the above question affirmatively by showing that there exist strongly infinitely divisible matrices that do not belong to \IM. For further properties of infinitely divisible and strongly infinitely divisible matrices, we refer the reader to \cite{SST2, van}.
The characterization of \SIDM\ in terms of the exponential map is given by the following theorem.

\begin{thm} {\rm \cite[Theorem 1]{van}}
			\label{SIDM}
		{\rm 	Matrix $A\in\mnr$ is \SIDM\ if and only if there exists an essentially nonnegative $B$ such that $A^t=e^{tB}$ for all $t\geq 0$. }
   \end{thm}

%For completeness and reference below, we review some more results from \cite{van}:

\begin{thm}\label{closerofIMpower}{\rm \cite[Theorem 2]{van}} 
		{\rm 	The set of infinitely divisible matrices contains the closure of the set $$\{A: A=K^n, n\in \mathbb{N}, K\in \IM\}.$$
			Futhermore, if $A$ is nonsingular, it is infinitely divisible if and only if it belongs in the closure of the set.}
\end{thm}

\begin{rem}
	{\rm 
The class of matrices of the form \(A=e^B\), where \(B\) is essentially nonnegative, arises naturally in the study of continuous-time Markov chains. The {\it embedding problem}, introduced in \cite{elfving}, asks whether a given stochastic matrix \(A\) can be realized as the transition matrix of a continuous-time Markov process. It was shown in \cite{Kingman} that a nonsingular stochastic matrix \(A\) is embeddable if and only if \(A=e^B\) for some essentially nonnegative matrix \(B\) satisfying \(Be=0\). Since every embeddable matrix is of this form, embeddable stochastic matrices constitute a subclass of \(\SIDM\). The following characterization of embeddability was established 
in \cite{Kingman}.
}
\end{rem}

\begin{thm}[{\rm \cite[Proposition 7]{Kingman}}]\label{embeddingprob}
{\rm Let \(A\in\mnr\) be a nonsingular stochastic matrix. Then \(A\) is embeddable if and only if, for each positive integer \(k\), there exists a stochastic matrix \(Q_k\) such that
\[
A=Q_k^k.
\]}
\end{thm}

The authors in \cite[Theorem 3.19]{SST2} have shown that a particular subclass of \SIDM\ matrices comprises $P$-matrices. To further this discussion, we will introduce the necessary and sufficient condition under which an \SIDM\ is a $P$-matrix. Before proceeding, let us recall the following theorem concerning when a principal submatrix of an \SIDM\ is \SIDM. 
\begin{thm} [{\rm \cite[Theorem 3.12]{SST2}}]
	        \label{prisubsidm}
		{\rm 	Let $A\in\mnr$ be \SIDM\ and $\alpha\subseteq\{1,2,\ldots,n\}$.
			Suppose that $A[\alpha]$ is an invertible principal submatrix of $A$ such that all 
			order roots of $A[\alpha]$ exist. Then $A[\alpha]$ is \SIDM.}
	\end{thm}
\begin{thm}
 {\rm Let $A\in\mnr$ be \SIDM. Then $A$ is a $P$-matrix if and only if all principal submatrices of $A$ are \SIDM.  }
\end{thm}
\begin{proof}
	For the forward implication, let us assume that $A\in\mnr$ is a $P$-matrix. In this case, all roots of the matrix $A$ exist. Since all principal submatrices of a $P$-matrix are also $P$-matrices, it follows that the roots of all principal submatrices of $A$ exist as well. Therefore, by Theorem \ref{prisubsidm}, we conclude that all principal submatrices of $A$ are 
	\SIDM. The converse follows from the fact that the determinant of an 
	\SIDM\ is positive. 
\end{proof}

Inverse $M$-matrices are contained in both the \SIDM\ and the class of $P$-matrices. In \cite{SST2}, it was shown that \IM\ and \SIDM\ are equivalent when the order of the matrix is $2$ \cite[Theorem 4.5]{SST2}. 
A reasonable conjecture that for $n\geq 3$,  
$\SIDM \;+\;\mbox{$P$-matrix} \implies \IM$
is not true, as demonstrated by the following example.
\begin{ex}
	{\rm 
 Consider $A=e^B=\begin{pmatrix}
  0.7847  &  0.7675  &  0.6493   \\
  3.7194  &  9.6403  &  3.7194\\ 
    1.4168 &   2.9519 &   1.5522
 \end{pmatrix}$  
 in which $B=\begin{pmatrix}
     -1   &  ~0   &  ~1\\
      ~1   &  ~2   &  ~1\\
       ~1  &   ~1   & -1
 \end{pmatrix}.$
 It can be verified that all principal minors of $A$ are positive, which means that
 $A$ is a $P$-matrix. However, the inverse of $A$ is
 $$A^{-1}=\begin{pmatrix}
     ~3.9842  &  ~0.7254 &  -3.4049\\
   -0.5034  &  ~0.2979  & -0.5034\\
   -2.6795  & -1.2287  &  ~4.7096
 \end{pmatrix}$$
 that is not a $Z$-matrix. Therefore, $A$ is not an inverse $M$-matrix. 
}
\end{ex}
%The question remains unresolved: 
%\begin{center}
%	{\bf Under what conditions does SIDM imply IM?}
%\end{center}

\begin{obs}\label{SIDMSchur}
	{\rm 
The Schur complements of the matrix $A$ in the above example are not \SIDM; moreover, they are not even nonnegative. Let $\alpha=\{3\}.$ Then, we have $$A/A[\alpha]=\begin{pmatrix}
        0.1919  & -0.4673\\
    0.3243  &  ~~~2.5668
    \end{pmatrix},$$ which is not nonnegative, and therefore, it is not an \SIDM.
}
\end{obs}

In the context of the relationship between \SIDM\ and \IM, Theorem~\ref{IMSchur1} and Observation~\ref{SIDMSchur} motivate the following question:

\begin{center}
{\bf When are all Schur complements of an \SIDM\ nonnegative?}
\end{center}

The next result gives a criterion for a strongly infinitely divisible matrix to be an inverse $M$-matrix.

\begin{thm}
{\rm Let $A\in\SIDM$. Then $A\in\IM$ if and only if $A$ is a $P$-matrix and all Schur complements of order $2$ are nonnegative.}
\end{thm}

\begin{proof}
Since $A\in\SIDM$, we have $A\geq0$. Moreover, if $A$ is a $P$-matrix, then all principal minors of $A$ are positive. Hence every Schur complement of order $1$ is positive, since each such Schur complement is the ratio of two positive principal minors. The result now follows from Theorem~\ref{IMSchur2}.
\end{proof}
\subsubsection{Doubly Nonnegative Matrices}
Integer powers of doubly nonnegative matrices remain doubly nonnegative. However, fractional powers, such as roots, may fail to preserve this property. We illustrate this with the following example.

\begin{ex}
{\rm Consider the matrix  
$
A =
\begin{bmatrix}
3 & 2 & 0 \\
2 & 3 & 1 \\
0 & 1 & 1
\end{bmatrix}.$
It is straightforward to verify that \(A\) is symmetric and entrywise nonnegative. Furthermore, since all its eigenvalues are positive, \(A\) is positive definite. Hence, \(A\) is a \emph{doubly nonnegative} (DN) matrix.

\medskip

The principal square root of \(A\) is approximately
\[
A^{1/2} \approx
\begin{bmatrix}
1.6001 & 0.6528 & -0.1160 \\
0.6528 & 1.5421 & 0.4425 \\
-0.1160 & 0.4425 & 0.8893
\end{bmatrix}.
\]
Although \(A^{1/2}\) is symmetric and positive definite, it contains negative entries; therefore, \(A^{1/2}\) is not doubly nonnegative.}
\end{ex}

\medskip

\noindent
The following corollary, which follows immediately from the characterization of \SIDM\ and Corollary~\ref{PDroots}, describes precisely when the fractional powers of an invertible doubly nonnegative matrix remain doubly nonnegative.
\begin{cor}
{\rm Let $A\in\mnr$ be an invertible \DN \ matrix. Then $A^p$ is doubly nonnegative for all $0<p\leq1$ if and only if $A\in\SIDM$.
In particular, the principal roots of $A$ are doubly nonnegative if and only if $A$ is strongly infinitely divisible.}
\end{cor}

\subsection{Integer Powers}

{We investigate whether various matrix positivity classes are preserved under positive integer powers. 
While some classes are closed under powers, others may lose their defining properties after repeated multiplication. 
The behavior depends strongly on the spectral and structural properties of the matrices involved.}

{Positive definite matrices are preserved under taking powers. 
However, the \(P\)-matrix property is not preserved in general; in fact, it may fail even for square powers. 
}

% \begin{thm}[{\cite[Theorem 4.1]{SST}}]
% {Let \(A\) be a \(P\)-matrix such that \(A^{m}\) is a \(P\)-matrix for all positive integers \(m\). 
% Then \(A^{t}\) is a \(P\)-matrix for all real numbers \(t\in \mathbb{R}\).}
% \end{thm}

% {Positive stability is also not generally preserved under integer powers. 
% If \(A\) is positive stable and \(\lambda\) is an eigenvalue of \(A\), then \(\lambda^k\) is an eigenvalue of \(A^k\). 
% Hence \(A^k\) remains positive stable if and only if}
% $
% \operatorname{Re}(\lambda^k)>0
% $
% {for every eigenvalue \(\lambda\) of \(A\). 
% Equivalently, the eigenvalues must satisfy}
% \[
% |\arg(\lambda)|<\frac{\pi}{2k}.
% \]
% {In particular, positive stability is preserved under all integer powers if and only if all eigenvalues of \(A\) are real and positive.}

{Positive stability is not generally preserved under integer powers. If $A$ is positive stable and $\lambda$ is an eigenvalue of $A$, then $\lambda^k$ is an eigenvalue of $A^k$. Hence $A^k$ is positive stable if and only if
\[
\operatorname{Re}(\lambda^k)>0
\]
for every eigenvalue $\lambda$ of $A$. Consequently, positive stability is preserved under the power map $A\mapsto A^k$ only under additional spectral restrictions, as characterized by the following theorem.

\begin{thm}
{\rm Let $A\in M_n(\mathbb{C})$ be positive stable. Then $A^k$ is positive stable for every $k\in\mathbb{N}$ if and only if every eigenvalue of $A$ is positive.}
\end{thm}

\begin{proof}
If every eigenvalue $\lambda$ of $A$ is positive, then $\lambda^k>0$ for all $k\in\mathbb{N}$. Since the eigenvalues of $A^k$ are precisely the numbers $\lambda^k$, it follows that $A^k$ is positive stable for every $k\in\mathbb{N}$.

Conversely, suppose that $A^k$ is positive stable for all $k\in\mathbb{N}$. Let $\lambda=re^{i\theta}$ be an eigenvalue of $A$. Then
\[
\operatorname{Re}(\lambda^k)=r^k\cos(k\theta)>0
\]
for all $k\in\mathbb{N}$. Hence $\cos(k\theta)>0$ for every $k\in\mathbb{N}$. This is possible only when $\theta=0$; otherwise, there exists $k\in\mathbb{N}$ such that $\cos(k\theta)\le 0$. Therefore every eigenvalue of $A$ is positive.
\end{proof}}

{Matrices whose Hermitian part is positive definite exhibit a similarly restrictive behavior. 
The following theorem provides a complete characterization.}

\begin{thm}[{\cite[Theorem~2]{Johnson1975}}]
{\rm Suppose \( A \in \Pi_n \).  
Then \( A^m \in \Pi_n \) for all positive integers \( m \) 
if and only if \( A \) is positive definite.}
\end{thm}

{Copositive matrices behave differently. 
Even powers of a copositive matrix are always positive semidefinite, whereas odd powers need not remain copositive. 
The following example illustrates this phenomenon.}

\begin{ex}{\rm 
Consider the Horn matrix
\[
H=
\begin{pmatrix}
1&-1&1&1&-1\\
-1&1&-1&1&1\\
1&-1&1&-1&1\\
1&1&-1&1&-1\\
-1&1&1&-1&1
\end{pmatrix},
\]
which is copositive. However,
\[
H^3=
\begin{pmatrix}
13&-11&5&5&-11\\
-11&13&-11&5&5\\
5&-11&13&-11&5\\
5&5&-11&13&-11\\
-11&5&5&-11&13
\end{pmatrix}
\]
is not copositive since, for
\(
x=\left(1,\frac32,1,0,0\right)^T\ge0,
\)
we obtain
\(
x^TH^3x=-\frac34<0.
\)
Hence \(H^3\) is not copositive although \(H\) itself is copositive.
}
\end{ex}

{Similarly, the defining properties of \(M\)-matrices, inverse \(M\)-matrices, \(H\)-matrices, inverse \(H\)-matrices, and semipositive matrices are not, in general, preserved under integer powers. We conclude with a characterization of those \(Z\)-matrices whose positive integer powers are irreducible \(M\)-matrices. Recall that a \(ZM\)-matrix is a matrix all of whose positive integer powers are \(Z\)-matrices, and an \(MMA\)-matrix is a matrix all of whose positive integer powers are irreducible \(M\)-matrices.}

\begin{thm}[{\cite[Theorem~3.9]{Friedland1987}}]
{\rm Let \(A\) be a \(Z\)-matrix. Then \(A\) is an \(MMA\)-matrix if and only if \(A\) is an irreducible \(ZM\)-matrix whose smallest eigenvalue \(\alpha_1\) is nonnegative.}
\end{thm}

% \TB{The preceding discussion is summarized in Table~\ref{tab:integerpowers}.}

% \begin{table}[h!]
% \centering
% \small
% \renewcommand{\arraystretch}{1.3}
% \setlength{\tabcolsep}{6pt}
% \begin{tabular}{||c|>{\raggedright\arraybackslash}p{5.5cm}|>{\raggedright\arraybackslash}p{6.5cm}||}
% \hline\hline
% \textbf{S. No} & \textbf{Matrix Class} & \textbf{Preserved under Integer Powers} \\
% \hline\hline

% 1 & Positive Definite (PD) & Yes \\
% \hline
% 2 & Doubly Nonnegative Matrices & Yes\\
% \hline
% 3 & \(P\)-matrices & Not always; may fail for \(A^2\) \\
% \hline

% 4 & Positive Stable Matrices & Not always \\
% \hline

% 5 & Matrices with Positive Definite Hermitian Part & iff \(A\) is positive definite \\
% \hline

% 6 & Copositive Matrices & Even powers only; odd powers may fail \\
% \hline

% 7 & \(M\)-matrices & Not always \\
% \hline

% 8 & Inverse \(M\)-matrices & Not always \\
% \hline

% 9 & Semipositive Matrices & Not always \\
% \hline\hline

% \end{tabular}
% \caption{Behavior of matrix positivity classes under integer powers}
% \label{tab:integerpowers}
% \end{table}

\section*{Conclusion}
{The preservation of matrix positivity properties under powers varies significantly across different matrix classes. While integer powers often maintain these properties in many classes, fractional powers may fail to do so, or vice versa. Interestingly, many matrix classes exhibit the property that fractional powers preserve the class, whereas integer powers generally do not. Based on whether a class is preserved under fractional powers, integer powers, both, or neither, we can categorize their behavior into four groups as shown below:}

\begin{table}[h]
\centering
\renewcommand{\arraystretch}{1.3}
\begin{tabular}{||p{6cm}||p{9cm}||}
\hline\hline
\textbf{Behavior Category} & \textbf{Matrix Classes} \\
\hline\hline
Fractional powers preserve, but integer powers do not & 
M-matrices, inverse M-matrices, positive stable matrices, matrices with positive definite Hermitian part, real H-matrices with positive diagonal entries, inverse H-matrices, semipositive matrices (fractional powers under a spectral condition)\\
\hline\hline
Integer powers preserve, but fractional powers do not & 
Entrywise nonnegative matrices, totally nonnegative matrices, doubly nonnegative matrices \\
\hline\hline
Both fractional and integer powers preserve & 
Positive definite matrices, SIDM matrices, Embeddable matrices \\
\hline\hline
Neither fractional nor integer powers preserve & 
Copositive matrices \\
\hline\hline
\end{tabular}
\caption{\footnotesize Classification of matrix classes based on preservation under fractional and integer powers.}
\end{table}

Building upon this classification, we investigated the behavior of several important matrix positivity classes under matrix functions induced by Pick functions, with particular emphasis on fractional powers and matrix roots.
Our analysis, built upon the theory of functional calculus and Pick functions, establishes unified criteria under which positivity and structural properties of matrices are preserved. More broadly, the framework highlights a deep link between analytic properties of Pick functions and algebraic structure preservation in matrix theory, thereby advancing the understanding of positivity-preserving transformations.

\vspace{0.5em}

A summary of the behavior of the main matrix classes studied under Pick functions and fractional powers is presented below.

\begin{table}[h!]
\centering
\small
\renewcommand{\arraystretch}{1.3}
\setlength{\tabcolsep}{6pt}
\begin{tabular}{||c|>{\raggedright\arraybackslash}p{4.2cm}|>{\raggedright\arraybackslash}p{4.7cm}|>{\raggedright\arraybackslash}p{5.2cm}||}
\hline\hline
\textbf{S. No} & \textbf{Matrix Class} & \textbf{Preserved by Pick Function} & \textbf{Preserved under Fractional Powers \( A^p \) ($0<p<1$)} \\
\hline\hline
1  & Positive Definite (\PD) & Yes & Yes \\
\hline
2  & Positive Stable & Yes & Yes \\
\hline
3  & Hermitian Part Positive Definite %(HPD) 
   & Yes & Yes \\
% \hline
% 4  & $P$-matrix & Yes & Yes \\
\hline
4  & Totally Positive (\TP) & No & No  \\
\hline
5  & Invertible $M$-matrix (\IM) & Yes & Yes \\
\hline
6  & Inverse $M$-matrix & Yes & Yes \\
\hline
7  & Semipositive %(SP) 
           & \textbf{Open} & Yes (under a spectral condition)\\
\hline
8  & $H$-matrix (with positive diagonal entries) & Yes & Yes \\
\hline
9 & Strictly Diagonally Dominant & Yes (with positive diagonal entries) & Yes (with positive diagonal entries) \\
\hline
10 & Copositive & If \( f[A] \) is defined, then Yes & If \( f[A] \) is defined, then Yes and \PD \\
\hline
11 & Nonnegative (entrywise) & No & No \\
\hline
12 & Strongly Infinitely Divisible Matrix (\SIDM) & \textbf{Open} & Yes \\
\hline
13 & Embeddable Matrix  & \textbf{Open} & Yes \\
\hline
14 & Doubly Nonnegative (\DN) & For \( n \ge 3 \), No & For \( n \ge 3 \), Yes if \( A \) is an \SIDM \\
\hline\hline
\end{tabular}

\vspace{0.5em}
\parbox{\textwidth}{\small\textit{Note:} The entries marked \textbf{Open} indicate cases for which preservation under Pick functions remains an open question.}

\caption{Behavior of Matrix Classes under Pick Functions and Fractional Powers}
\end{table}
\vspace{0.5em}

\section*{Acknowledgments}

 The work of PIMS Postdoctoral Fellow S.\ Mondal leading to this publication was supported in part by the Pacific Institute for the Mathematical Sciences.

\end{document}